\documentclass{amsart}

\usepackage{amsmath,amssymb,amsthm,a4wide}

\newtheorem{theorem}{Theorem}[section]
\newtheorem{remark}[theorem]{ Remark}

\newtheorem{corollary}[theorem]{Corollary}

\newtheorem{proposition}[theorem]{Proposition}

\newtheorem{definition}[theorem]{Definition}

\newcommand{\RR} {\mathbb R}

\newcommand{\pa} {\partial}

\newcommand{\beq} {\begin{equation}}
\newcommand{\eeq} {\end{equation}}

\newcommand{\inrad}{\operatorname{inrad}}

\begin{document}

\title[ 
Remarks on a result of Harrell-Kr\"{o}ger-Kurata ]{ On maximizing the fundamental frequency of the complement of an obstacle}

\author[]{Bogdan Georgiev}
\address{Max Planck Institute for Mathematics, Vivatsgasse 7, 53111 Bonn, Germany}
\email{bogeor@mpim-bonn.mpg.de}

\author[]{Mayukh Mukherjee}
\address{Mathematics Department, Technion - I.I.T., Haifa 32000, Israel}
\email{mathmukherjee@gmail.com}

\begin{abstract}
	
	Let $\Omega \subset \RR^n$ be a bounded domain satisfying a Hayman-type asymmetry condition, and let $ D $ be an arbitrary bounded domain referred to as "obstacle". We are interested in the behaviour of the first Dirichlet eigenvalue $ \lambda_1(\Omega \setminus (x+D)) $.

	First, we prove an upper bound on $ \lambda_1(\Omega \setminus (x+D)) $ in terms of the distance of the set $ x+D $ to the set of maximum points $ x_0 $ of the first Dirichlet ground state $ \phi_{\lambda_1} > 0 $ of $ \Omega $. In short, a direct corollary is that if
	\begin{equation}
		\mu_\Omega := \max_{x}\lambda_1(\Omega \setminus (x+D))
	\end{equation}
	is large enough in terms of $ \lambda_1(\Omega) $, then all maximizer sets $ x+D $ of $ \mu_\Omega $ are close to each maximum point $ x_0 $ of $ \phi_{\lambda_1} $.
	
	Second, we discuss the distribution of $ \phi_{\lambda_1(\Omega)} $ and the possibility to inscribe wavelength balls at a given point in $ \Omega $.
	
	Finally, we specify our observations to convex obstacles $ D $ and show that if $ \mu_\Omega $ is sufficiently large with respect to $ \lambda_1(\Omega) $, then all maximizers $ x+D $ of $ \mu_\Omega $ contain all maximum points $ x_0 $ of $ \phi_{\lambda_1(\Omega)} $. 
%
\end{abstract}


\maketitle

\section{Introduction and background}
	We consider the problem of placing of an obstacle in a domain so as to maximize the fundamental frequency of the complement of the obstacle. To be more precise, let $\Omega \subset \RR^n$ be a bounded domain, and let $ D $ be another bounded domain referred to as "obstacle". The problem is to determine the optimal translate $ x+D $ so that the fundamental Dirichlet Laplacian eigenvalue $\lambda_1(\Omega \setminus (x + D))$ is maximized/minimized.
	
	In case the obstacle $ D $ is a ball, physical intuition suggests that for sufficiently regular domains and sufficiently small balls, $\Omega$, $\lambda_1(\Omega \setminus B_r(x))$ will be maximized when $x = x_0$, a point of maximum of the ground state Dirichlet eigenfunction $\phi_{\lambda_1}$ of $\Omega$. Heuristically, such maximum points $ x_0 $ seem to be situated deeply in $ \Omega $, hence removing a ball around $ x_0 $ should be an optimal way of truncating the lowest possible frequency. Our methods give equally good results for Schr\"{o}dinger operators on a large class of bounded domains sitting inside Riemannian manifolds (see the remarks at the end of Section \ref{sec:mainres}).

	The following well-known result of Harrell-Kr\"{o}ger-Kurata treats the case when $ \Omega $ satisfies convexity and symmetry conditions:
	
	\begin{theorem}[\cite{HKK}]\label{thm:HKK}
		Let $\Omega$ be a convex domain in $\RR^n$ and $B$ a ball contained in $\Omega$. Assume that $\Omega$ is symmetric with respect to some hyperplane $H$. Then,
	\newline
	(a) at the maximizing position, $B$ is centered on $H$, and \newline
	(b) at the minimizing position, $B$ touches the boundary of $\Omega$.
	\end{theorem}
	
	The last result of Harrell-Kr\"{o}ger-Kurata seems to work under rather strong symmetry assumption. We also recall that the proof of Harrell-Kr\"{o}ger-Kurata proceeds via a moving planes method which essentially measures the derivative of $\lambda_1 (\Omega \setminus B)$ when $B$ is shifted in a normal direction to the hyperplane (also see pp 58 of \cite{He}).
	
	There does not seem to be any result in the literature treating domains without symmetry or convexity properties.
	
	In our note, we consider bounded domains $\Omega \subset \RR^n$ which satisfy an asymmetry assumption in the following sense:
	\begin{definition}\label{def:AA}
		A bounded domain $\Omega \subset \RR^n$ is said to satisfy the asymmetry assumption with coefficient $\alpha$ (or $ \Omega $ is $ \alpha $-asymmetric) if for all $x \in \partial \Omega$, and all $r_0 > 0$, 
	 	\begin{equation}
		 	\frac{|B_{r_0}(x) \setminus \Omega|}{|B_{r_0}(x)|} \geq \alpha.
	 	\end{equation}
	\end{definition}
	This condition seems to have been introduced in \cite{Ha}. Further, the $ \alpha $-asymmetry property was utilized by D. Mangoubi in order to obtain inradius bounds for Laplacian nodal domains (cf. \cite{M}) as nodal domains are asymmetric with $ \alpha = \frac{C}{\lambda^{(n-1)/2}} $.
	
	From our perspective, the notion of asymmetry is useful as it basically rules out narrow "spikes" (i.e. with relatively small volume) entering deeply into $\Omega$. For example, let us also observe that convex domains trivially satisfy our asymmetry assumption with coefficient $\alpha = \frac{1}{2}$.

	\section{The basic estimate for general obstacles} \label{sec:mainres}

	With the above in mind, we consider any bounded $ \alpha$-asymmetric domain $\Omega \subset \RR^n$ and a bounded obstacle domain $ D $. We denote the first positive Dirichlet eigenvalue and eigenfunction of $ \Omega $ by $\lambda_1 $ and $ \phi_{\lambda_1(\Omega)} $ respectively and let
	\begin{equation}
		M := \{ x \in \Omega ~|~ \phi_{\lambda_1}(x) = \| \phi_{\lambda_1(\Omega)} \|_{L^\infty (\Omega)} \}
	\end{equation}
	be the set of maximum points of $ \phi_{\lambda_1 (\Omega)} $.
	
	Let us also put
	\begin{equation}
		\mu_\Omega := \max_{x} \lambda_1(\Omega \setminus (x + D)).
	\end{equation}
	Finally, for a given translate $ x+D $ of the obstacle let us set
	\begin{equation}
		\rho_x := \max_{y \in M} d(y, x+D),
	\end{equation}
	measuring the maximum distance from a maximum point of $ \phi_{\lambda_1(\Omega)} $ to the translate $ x+D $.
	
	Our main estimate is the following.
	\begin{theorem}\label{thm:mainres}
		Let us fix a translate $ (x + D) $ and assume that $  \rho_x > 0  $. Then
		\begin{equation}
			\lambda_1(\Omega \setminus (x + D)) \leq \beta(\rho_x) \lambda_1 (\Omega),
		\end{equation}
		where $ \beta $ is a continuous decreasing function defined as
		\begin{equation}
			\beta(\rho) =
			\begin{cases}
				\beta_0 = \beta_0(n, \alpha), \quad \rho \sqrt{\lambda_1(\Omega)} > r_0 := r_0(n, \alpha), \\
				\frac{c_0}{\rho^2 \lambda_1(\Omega)}, \quad \rho \sqrt{\lambda_1(\Omega)} \leq r_0, \quad c_0 = c_0 (n),
			\end{cases}
		\end{equation}
		where $\beta_0 r_0 = c_0$.
	\end{theorem}	
	
	We remark that in particular if $ \rho_x $ is of sub-wavelength order (i.e. $ \lesssim \frac{1}{\sqrt{\lambda_1(\Omega)}} $), then $ \lambda_1(\Omega \setminus (x + D)) \lesssim \frac{1}{\rho_x^2} $. If the obstacle $D$ is convex, we can say more (see Theorem \ref{thm:Kurata-comment} below).
	
	\begin{proof}[Proof of Theorem \ref{thm:mainres}]
		The proof essentially exploits the fact that there are ``almost inscribed'' wavelength balls centered at maximum points of $ \phi_{\lambda_1(\Omega)} $. To make this statement precise, we recall the following theorem from \cite{GM}, which works for {\bf all domains} in compact Riemannian manifolds of dimension $n \geq 3$ (planar domains are known to have wavelength inradius from the work of Hayman (\cite{Ha})):
		
		\begin{theorem} \label{th:Large-Inscribed-Ball}
			Let $\dim M \geq 3, \epsilon_0 > 0 $ be fixed, $\Omega$ a domain inside $M$, and $ x_0 \in \Omega$ be such that $ |\varphi_\lambda(x_0)| = max_{\Omega}|\varphi_\lambda| $, where $\varphi_\lambda$ is the ground state Dirichlet eigenfunction of $\Omega$. There exists $ r_0 = r_0 (\epsilon_0) $, such that
			\begin{equation}\label{Vol}
			\frac{| B_{r_0} \cap \Omega |}{|B_{r_0}|} \geq 1 - \epsilon_0,
			\end{equation}
			where $ B_{r_0} $ denotes $ B\left(x_0, \frac{r_0}{\sqrt{\lambda_1}}\right) $.
		\end{theorem}
		We also note that it follows from the proof that $r_0$ can be taken as $r_0 = \epsilon_0^{\frac{n - 2}{2n}}$. Moreover, let us for completeness recall that Theorem $ \ref{th:Large-Inscribed-Ball} $ relies on two main ingredients - namely, the Feynman-Kac formula and certain capacity estimates related to hitting probabilities of Brownian motion. We refer to \cite{GM} and \cite{RS} for more details.
		
		Now, it is clear that under the $ \alpha $-asymmetry assumption, there exists an $ r_0 := r_0(\alpha, n) $, such that around each maximum point $ x_0 \in \Omega $ of $ \phi_{\lambda_1(\Omega)} $ one can find a fully inscribed ball $B_{r_0 / \sqrt{\lambda_1(\Omega)}} (x_0) \subseteq \Omega$.
		By the definition of $ \rho_x $ it follows that we can find a maximum point $ x_0 \in  (\Omega \setminus (x + D)) $ and an inscribed ball $ B_{\rho_0}(x_0) $ where
		\begin{equation}
		\rho_0 := \min \left( \frac{r_0}{\sqrt{\lambda_1(\Omega)}}, \rho_x \right).
		\end{equation}
		
		As the first eigenvalue is monotonic with respect to inclusion, we see that
		
		\begin{equation}
		\lambda_1(\Omega \setminus (x + D)) \leq \lambda_1(B_{\rho_0}(x_0)) = \frac{C}{\rho_0^2},
		\end{equation}
		
		where $ C = C(n) $ is a universal constant.
		
		
		Expressing the right hand side of the last inequality in terms of $ \lambda_1(\Omega) $ we define the function $ \beta(\rho) $ as above.
		
		This concludes the proof.
	\end{proof}
	
	Here, we have considered the obstacle problem in the case of Euclidean spaces, on reasonably well-behaved domains, and for the operator $-\Delta + \lambda_1(\Omega)$, as that seems to be the primary case of interest. However, we also include some remarks outlining some straightforward generalizations.
	
	\begin{remark}
		It is clear that removing capacity zero sets from $ \alpha $-asymmetric domains considered in Definition \ref{def:AA} will lead to the same conclusions. Indeed, in this situation we will not be dealing with fully inscribed balls as above - instead, we will have balls whose first eigenvalue is comparable to the one of an inscribed one.
	\end{remark}
	\begin{remark}
		Also, in the setting of curved spaces, one has absolutely similar results for $\Omega \subseteq M$, where $(M, g)$ is a smooth compact Riemannian manifold, if we allow the constants to depend on the dimension, asymmetry and the metric $g$.
	\end{remark}
	\begin{remark}
		Lastly, it is clear that the results of \cite{RS} allow us to extend our discussion here from operators of the form $-\Delta + \lambda_1(\Omega)$ to Schr\"{o}dinger operators of the form $-\Delta + V$, where $V$ is bounded above. The conclusions are analogous with $\lambda_1(\Omega)$ replaced by $\| V\|_{L^\infty}$and the proofs are identical. 
	\end{remark}
	
	Now, as an immediate implication of Theorem \ref{thm:mainres} we have the following corollary.
	
	\begin{corollary}\label{cor:Localization}
		Suppose that $ \mu_\Omega = C_0 \lambda_1(\Omega)$, where $ C_0 > \frac{c_0}{r_0^2} $ is a given fixed constant and $ c_0, r_0 $ are the constants in Theorem $ \ref{thm:mainres} $. Then, for a maximizer $ \bar{x} + D $ of $ \mu_\Omega $ we have
		\begin{equation}
			\rho_{\bar{x}} \leq \beta^{-1}(C_0).
		\end{equation}
		
		In particular, if $ C_0 $ is large,
		\begin{equation}
			\rho_{\bar{x}} \lesssim \frac{1}{\sqrt{C_0 \lambda_1(\Omega)}}.
		\end{equation}
	\end{corollary}
	
	In other words the above corollary can be interpreted as follows: either $ \mu_\Omega $ is comparable to $ \lambda_1(\Omega) $, or the maximum points of $ \phi_{\lambda_1(\Omega)} $ are near the maximizer sets $ \bar{x} + D $ of $ \mu_\Omega $.
	
	We note that the localization in the Corollary above gets better when $ C_0 $ is large. By Faber-Krahn's inequality, straightforward examples with large $ C_0 $ are domains $ \Omega $ for which $ |\Omega \setminus (x +D) | $ is sufficiently small for some $ x $.

	Particularly, for bounded convex domains in $\RR^n$, by a theorem of Brascamp-Lieb, the level sets of $\phi_{\lambda_1(\Omega)}$ are convex. Since $\phi_{\lambda_1(\Omega)}$ is real analytic and it can be assumed positive on $\Omega \setminus \pa \Omega$ without loss of generality, this means that it has a unique point of maximum. So, in this setting, our result heuristically says that if removal of a ball $B_r$ has a ``significant effect'' on the vibration of $\Omega \setminus B_r$, then $B_r$ must be centered quite close to the max point of the ground state Dirichlet eigenfunction $\phi_{\lambda_1}$ of the domain $\Omega$, where the bound on $\rho_x$ gives the quantitative relation between the ``effect'' and the order of ``closeness''. In a sense, this can be seen to be complementary to Corollary II.3 of \cite{HKK}. 

	\section{Inscribed balls and distribution of $ \phi_{\lambda_1(\Omega)} $}

	Further, we specify our results to the obstacle being a ball $ D $. We point out a few statements related to the connection between the distribution of $ \phi_{\lambda_1(\Omega)} $ and the possibility to inscribe a large ball at a given point $ x $ in $ \Omega $.
	
	First, by Theorem \ref{th:Large-Inscribed-Ball} above we immediately have the following observation:
	
	\begin{proposition}\label{prop:Ball-at-Max}
		Let $ \Omega $ be $ \alpha $-asymmetric and let $ \inrad(\Omega) $ denote the inner radius of $ \Omega $. If $ x_0 $ is a point of maximum of $ \phi_{\lambda_1(\Omega)} $, then there exists an inscribed ball $ B_{C \inrad(\Omega)}(x_0) \subseteq \Omega $, where $ C = C(n, \alpha) $.
	\end{proposition}
	
	\begin{proof}[Proof of Proposition \ref{prop:Ball-at-Max}]
		We observe that by the results of \cite{M}, $ \alpha $-asymmetric domains $ \Omega $ satisfy
		\begin{equation}
			\frac{C_1(\alpha, n)}{\sqrt{\lambda_1(\Omega)}} \leq \inrad(\Omega) \leq \frac{C_2(n)}{\sqrt{\lambda_1(\Omega)}}.
		\end{equation}
			
		Now, it follows from our Theorem \ref{th:Large-Inscribed-Ball} (see \cite{GM}) that there exists an inscribed wavelength ball at the max point $ x_0 $, which concludes the proof.
	\end{proof}
	
	In particular, the last proposition applies for convex domains. We mention in this connection that localization results for maximum points of $ \phi_{\lambda_1(\Omega)} $ in case $ \Omega  $ is a planar convex domain can be found in the work of Grieser-Jerison (see \cite{GJ}).

	On the other hand, it is natural to ask how large is $ \phi_{\lambda_1(\Omega)} $ at points admitting a large inscribed ball. For reasonably nicely behaved domains, we have the following:
	
	\begin{corollary} \label{cor:Lower-Bound}
		Let $ \Omega $ be a $C^{2, \beta}$-regular $\alpha $-asymmetric domain and let $ \phi_{\lambda_1(\Omega)} $ be normalized so that $ \| \phi_{\lambda_1(\Omega)} \|_{L^\infty(\Omega)} = 1 $. Suppose that for $ \tilde{x} \in \Omega $ there exists a maximal inscribed ball $ B_r(\tilde{x}) \subseteq \Omega $ where $ r := c \inrad(\Omega) $ for some $0 < c \leq 1 $, such that $ \frac{|\Omega \setminus B_r(\tilde{x})|}{|\Omega|} $ is sufficiently small. Then
		\begin{equation}
			\phi_\lambda (\tilde{x}) > C,
		\end{equation}
		where $ C = C(|\Omega|, \pa \Omega, c, n) $.
	\end{corollary}
	
	Analogously, one can show a similar statement by demanding that $ |B_r(\tilde{x}) \cap \Omega| $ is sufficiently large in comparison to $ |\Omega| $.
	
	\begin{proof}[Proof of Corollary \ref{cor:Lower-Bound}]
		Let us first suppose that
		\begin{equation}
		|\Omega| = \kappa r^n, \quad \kappa > \omega_n,
		\end{equation}
		where $\omega_n$ is the volume of a ball of radius $1$. 
		We use the Faber-Krahn inequality to obtain
		\begin{multline}
		\lambda_1(\Omega \setminus B_r (\tilde{x})) \geq \frac{C}{|\Omega \setminus B_r (\tilde{x})|^{2/n}} = \frac{C}{(|\Omega| - \omega_nr^n)^{2/n}} = \frac{C}{(\kappa-\omega_n)^{2/n} r^2} = \\ = \frac{C}{(\kappa-\omega_n)^{2/n} (c \inrad(\Omega))^2} \geq \frac{C C_2(n)}{c^2 (\kappa-\omega_n)^{2/n}} \lambda_1(\Omega) =: \tilde{C}_0 \lambda_1(\Omega).
		\end{multline}
		
		By assumption, $ \tilde{C}_0 $ is sufficiently large, i.e., in particular $ \tilde{C}_0 > \frac{c_0}{r^2_0} $, so we may apply Corollary \ref{cor:Localization} to obtain that
		\begin{equation}
		\rho_{\tilde{x}} \leq \beta^{-1}(\tilde{C}_0) = \sqrt{\frac{c_0}{\tilde{C}_0 \lambda_1(\Omega)}}.
		\end{equation} 
		
		On the other hand, the Schauder a priori estimates up to the boundary for $ \phi_{\lambda_1(\Omega)} $ (see \cite{GT}, Theorem 6.6) yield the existence of $ \gamma = \gamma(\Omega, n) $, such that
		\begin{equation}
		\| \nabla \phi_{\lambda_1(\Omega)} \|_{L^\infty(\Omega)} \leq \gamma(\Omega, n) \sqrt{\lambda_1(\Omega)}.
		\end{equation}
		
		As by assumption $ \phi_{\lambda_1(\Omega)} (x_0) = 1 $ and $ \tilde{C}_0 $ is sufficiently large, then
		\begin{equation}
		\phi_{\lambda_1(\Omega)}(\tilde{x}) \geq C = C(c_0, \tilde{C}_0, \gamma),
		\end{equation}
		which concludes the claim.
	\end{proof}
	
	\section{Relation between maximum points and convex obstacles}
	
		Note that Theorem \ref{thm:mainres} holds for arbitrary obstacles and gives a bound on the distance $ \rho_x $ to maximum points of $ \phi_{\lambda_1(\Omega)} $. However, it is desirable to deduce that $ \rho_x = 0 $, i.e. maximizers actually contain the maximum points of $ \phi_{\lambda_1(\Omega)} $.
		
		From Proposition \ref{prop:Ball-at-Max} and Theorem \ref{thm:mainres} we deduce the following:
		
		\begin{theorem} \label{thm:Kurata-comment}
			Let $D$ be a convex obstacle, and $\bar{x} + D$ maximize $\lambda_1(\Omega \setminus (x + D))$. Then there exists a constant $C_0 = C_0(\alpha, n)$ such that if $\lambda_1(\Omega \setminus (\bar{x} + D)) \geq C \lambda_1(\Omega) $ for some $C \geq C_0$, then $\rho_{\bar{x}} = 0$.
			
			In other words, either $ \mu_\Omega \sim \lambda_1 (\Omega) $ or $ \rho_{\bar{x}} = 0 $.
			\end{theorem}
			
			\begin{proof}
				To the contrary let us suppose that $ \rho_{\bar{x}} = d(\bar{x} + D, x_0) > 0 $ where $ x_0 $ is a maximum point of $ \phi_{\lambda_1(\Omega)} $ and $ \lambda_1(\Omega \setminus (\bar{x} + D)) \geq C \lambda_1(\Omega) $ for an arbitrary large $ C > 0 $.
				
				We apply the statement of Proposition \ref{prop:Ball-at-Max} and deduce that there is a wavelength inscribed ball $ B $ at $ x_0 $. As $ D $ is a convex domain, we can find a wavelength half-ball $ B^{1/2} \subset \Omega \setminus (\bar{x}+D) $ containing $ x_0 $. By the assumption and eigenvalue monotonicity with respect to inclusion:
				
				\begin{equation}
					 C \lambda_1(\Omega) \leq \lambda_1(\Omega \setminus (\bar{x} + D)) \leq \lambda_1(B^{1/2}) \leq \frac{C_1}{(\inrad (\Omega))^2} = C_2 \lambda_1(\Omega),
				\end{equation}
				where $  C_2 = C_2(n, \alpha) $. Taking $ C $ sufficiently large we get a contradiction.
				
			\end{proof}
			
   It is clear that for explicit applications, particularly in the case of convex domains, Theorem \ref{thm:Kurata-comment} is dependent on a precise knowledge of the location of the maximum point of $\phi_{\lambda_1(\Omega)}$. Localization of the maximum point of $\phi_{\lambda_1(\Omega)}$ (or more generally, the ``hot spot'') is a problem which is far from being settled. Here we take the space to augment Theorem \ref{thm:Kurata-comment} with the recent results of \cite{BMS}.
   
    First we recall the definition of the ``heart'' of a convex body $\Omega$. The following intuitive definition appears in \cite{EH}, and it is equivalent to the (more technical) definition presented in \cite{BMS}.
   
	   \begin{definition}\label{def:heart-conv-dom}
	   	Let $P$ be a hyperplane in $\RR^n$ which intersects $\Omega$ so that $\Omega \setminus P$ is the union of two components located on either side of $P$. The domain $\Omega$ is said to have the interior reflection property with respect to $P$ if the reflection through $P$ of one of these subsets, denoted $\Omega_s$, is contained in $\Omega$, and in that case $P$ is called a hyperplane of interior reflection for $\Omega$. 
	   	When $\Omega$ is convex, the heart of $\Omega$, denoted by $\heartsuit (\Omega)$, is defined as the intersection of all such $\Omega\setminus\Omega_s$ with respect to hyperplanes of interior reflection of $\Omega$.
	   \end{definition}
	   
	The following result is contained in Proposition 4.1 of \cite{BMS}.
	\begin{proposition}[\cite{BMS}]\label{prop:BMS}
	The unique maximum point $x_0$ of $\phi_{\lambda_1(\Omega)}$ is contained in $\heartsuit(\Omega)$. Furthermore, $x_0$ is contained in the interior of $\heartsuit(\Omega)$, if the latter is non-empty.
	\end{proposition}

	
	
	\subsection*{Acknowledgements} We thank Saskia Roos for drawing our attention to the reference \cite{He}. We are grateful to Antoine Henrot and Kazuhiro Kurata for their comments on a draft version. We also gratefully acknowledge the Max Planck Institute for Mathematics, Bonn and the Technion, Haifa for providing ideal working conditions.

\end{document}